\documentclass[11pt, a4paper]{amsart}
\usepackage{cgeometry}
\newcommand{\im}{\operatorname{Im}}
\newcommand{\Jet}{\operatorname{Jet}}
\newcommand{\rmax}{\operatorname{max}}
\newcommand{\rk}{\operatorname{rk}}

\title{Analytic Bertini theorem II --- The local case}
\author{Mingchen Xia}
\date{}

\begin{document}
\begin{abstract}
  We prove the local analytic Bertini theorem, confirming a conjecture of Boucksom in full generality.
\end{abstract}
\maketitle
 \setcounter{tocdepth}{1}
\tableofcontents

\section{Introduction}
For a non-negative integer $k$, write $\Delta^k$ for the unit polydisc of dimension $k$. 
The main result of the paper is the following:
\begin{theorem}[Local analytic Bertini theorem]\label{thm:main}
Let $m,n$ be non-negative integers and
\begin{equation}\label{eq:X_Phi}
 X=\Delta_z^m\times\Delta_\eta^n, \qquad \Phi\in\PSH(X).
\end{equation}
For $\eta\in\Delta_\eta^n$, put $X_\eta\coloneqq \Delta_z^m\times\{\eta\}$.
There is a pluripolar set $P\subseteq\Delta_\eta^n$ such that, for every $\eta\notin P$, the multiplier ideal sheaves satisfy
\begin{equation}\label{eq:abt}
 \mathcal{I}_X(\Phi)\cdot\mathcal{O}_{X_\eta}
   =\mathcal{I}_{X_\eta}\left(\Phi|_{X_\eta}\right)
\end{equation}
as ideal sheaves on $X_\eta$.
\end{theorem}
Here and throughout, $\mathcal{I}_X(\Phi)\cdot\mathcal{O}_{X_\eta}$ denotes the image ideal in $\mathcal{O}_{X_\eta}$. 

The inclusion $\supseteq$ in \eqref{eq:abt} is a simple consequence of the Ohsawa--Takegoshi extension theorem. Furthermore, Fubini's theorem shows that the exceptional set of $\eta$ is Lebesgue null; related restriction results were obtained by Fujino--Matsumura and Meng--Zhou \cite{FM21,MZ23}. 
Boucksom asked whether the exceptional set can be chosen pluripolar. Previously, for a projective fibration, the author confirmed this conjecture in \cite{XiaBer}. The proof relies on the technique of positivity of direct images of Berndtsson \cite{Bern09,HPS18} and does not generalize immediately to the local setting.

In the local setting, the main difficulty is that the direct image of $\mathcal{I}(\Phi)$ on $\Delta^n_{\eta}$ does not admit any natural geometric structure\footnote{It is a quasi-coherent sheaf, once the language is properly set up, but this does not seem very useful for the present problem.}, due to the non-compactness of the fibers. We shall replace the
direct image argument used in the global setting by a jet bundle argument, and Berndtsson's positivity theorem by the log-plurisubharmonicity theorem for fiberwise weighted Bergman kernels due to Bao--Guan--Yuan \cite{BGY23variant}, extending Berndtsson's theorem \cite{Ber06}. The idea of the proof is very close to the global setting, though the technical details are much more difficult.

As an intermediate step, we prove \cref{thm:plurifine-zero-set,cor:projected-zero-set}, which upgrade a Lebesgue null set to a pluripolar set. Their proofs rely crucially on the fine pluripotential theory developed by El Kadiri, Fuglede, Wiegerinck, El Marzguioui, and others. It is also interesting to observe that although the final result \cref{thm:main} is in the local setting, the proof of this intermediate step requires the positivity of direct images in the global setting.

\Cref{thm:main} can also be easily extended to a general fibration, as we handle in \cref{sec:manifold-bertini}.
\begin{corollary}\label{cor:manifold-bertini}
Let $f\colon Y\to Z$ be a morphism of ($\sigma$-compact) complex manifolds and let $\Phi\in\PSH(Y)$. There is a (locally) pluripolar set $P\subseteq Z$ such that, for every $z\in Z\setminus P$, the fiber $Y_z\coloneqq f^{-1}(z)$ is a (possibly empty) complex manifold and
\begin{equation}\label{eq:abt-manifold}
 \mathcal I_Y(\Phi)\cdot\mathcal O_{Y_z}
 =\mathcal I_{Y_z}(\Phi|_{Y_z}).
\end{equation}
\end{corollary}

\Cref{thm:main} and \cref{cor:manifold-bertini} are the first step towards understanding the deformations of psh singularities, a project which the author wishes to develop in the near future.

\subsection*{Acknowledgments}
AI plays a key role in this project. The idea of the proof was due to the author before the AI era. The details were, however, first carried out by the \href{https://github.com/frenzymath/Rethlas}{Rethlas agent} using the OpenAI GPT-5.6-sol model. The AI-generated proofs were later simplified and largely rewritten by the author.

The author would like to thank Charles Favre for accidentally bringing the book \cite{EKF} to his attention, which turns out to be indispensable for this problem.

The author is supported by the National Key R\&D Program of China 2025YFA1018200.

\section{Preliminaries}

\subsection{Multiplier ideals and the almost-everywhere case}
Given a psh function $\varphi$ on a complex manifold $Y$, the multiplier ideal sheaf $\mathcal{I}(\varphi)=\mathcal{I}_Y(\varphi)$ consists of the holomorphic germs $f$ for which $|f|^2\mathrm{e}^{-\varphi}$ is locally integrable. By Nadel's coherence theorem, $\mathcal{I}_Y(\varphi)$ is coherent. We use the convention $\mathcal{I}_Y(-\infty)=0$.

Given a function $f(z,\eta)$ depending on two variables, we shall write $f_{\eta}$ for the function sending $z$ to $f(z,\eta)$.

Fix $m,n$, $X$ and $\Phi$ as in~\eqref{eq:X_Phi}. For each $\eta\in\Delta_\eta^n$, write
\begin{equation}\label{eq:JJetaIeta}
 J\coloneqq\mathcal{I}_X(\Phi),\qquad J_\eta\coloneqq J\cdot\mathcal{O}_{X_\eta},\qquad I_\eta\coloneqq\mathcal{I}_{X_\eta}(\Phi_\eta).
\end{equation}
The Ohsawa--Takegoshi extension theorem gives
\begin{equation}\label{eq:OT-inclusion}
 I_\eta\subseteq J_\eta.
\end{equation}

The reverse inclusion holds on almost every fiber by Fubini's theorem.
Choose a countable cover of $X$ by product charts
$U'=W'\times G'\Subset U=W\times G\Subset X$, small enough that
$J|_U$ admits finitely many generators. Their weighted squares are
integrable on $U'$. For each chart, Fubini's theorem gives a null subset
of $G'$ outside which all restricted generators belong to $I_\eta$ on
$W'$. Thus $J_\eta\subseteq I_\eta$ on $W'$. Taking the countable union
of these null sets and using~\eqref{eq:OT-inclusion} gives
\begin{equation}\label{eq:ae-good}
 G_{\mathrm{ae}}\coloneqq \left\{\eta\in\Delta_\eta^n:I_\eta=J_\eta \right\}, \qquad |\Delta_\eta^n\setminus G_{\mathrm{ae}}|=0,
\end{equation}
where $|\cdot|$ denotes Lebesgue measure. Compare~\cite{MZ23} for an analogous restriction result.

\subsection{Weighted Bergman kernels}

We use the convention $\mathbb N=\mathbb Z_{\geq0}$. For $k\in \mathbb{N}$ and $\alpha=(\alpha_1,\ldots,\alpha_k)\in\mathbb N^k$, write
\[
 |\alpha|\coloneqq \sum_{j=1}^k\alpha_j,
 \qquad
 \alpha!\coloneqq \prod_{j=1}^k\alpha_j!,
 \qquad
 \partial_u^\alpha\coloneqq
 \prod_{j=1}^k \frac{\partial^{\alpha_j}}{\partial u_j^{\alpha_j}},
 \qquad
 D_u^\alpha\coloneqq \frac{1}{\alpha!}\partial_u^\alpha.
\]
If $U\subseteq\mathbb C_u^k$ is a domain and $\varphi\in\PSH(U)$, define
\[
 A^2(U,\mathrm{e}^{-\varphi})\coloneqq
 \left\{f\in\mathcal{O}(U):
 \|f\|_{U,\varphi}^2\coloneqq
 \int_{U}|f|^2\mathrm{e}^{-\varphi}\,\mathrm dV<\infty\right\}.
\]
Here and in the sequel, $\mathrm dV$ denotes standard Lebesgue measure.

Let $q$ and $s$ be positive integers. Given finitely supported coefficients
$(c_\alpha)_{\alpha\in\mathbb N^q}$ in $\mathbb C$, put
\[
 \mathcal D\coloneqq\sum_{\alpha\in\mathbb N^q}c_\alpha D_u^\alpha.
\]

Let $V\subseteq\mathbb C_u^q$ be a bounded pseudoconvex domain, let $G\subseteq\mathbb C_\eta^s$ be a domain, and let $\psi\in\PSH(V\times G)$.
For $\eta\in G$, consider the weighted Bergman kernels:
\[
 K_{\mathcal D,V}^{\psi_\eta}(u)\coloneqq
 \sup_{0\ne f\in A^2(V,\mathrm{e}^{-\psi_\eta})}\frac{|(\mathcal Df)(u)|^2}{\int_V|f|^2\mathrm{e}^{-\psi_\eta}\,\mathrm dV},
 \qquad u\in V.
\]
If the space $A^2(V,\mathrm{e}^{-\psi_\eta})$ is trivial, we understand that $K_{\mathcal D,V}^{\psi_\eta}\equiv 0$.

We use the following special case of
\cite[Theorem~1.4]{BGY23variant}, which extends Berndtsson's theorem
\cite{Ber06}.
\begin{theorem}\label{thm:fixed-functional}
With the notations above,
\[
 \left[(u,\eta)\longmapsto\log K_{\mathcal D,V}^{\psi_\eta}(u)\right]
 \in\PSH(V\times G)\cup\{-\infty\}.
\]
\end{theorem}
See also~\cite{BG23,BGY22}.

\section{A pluripolar zero-set criterion}

Let $\mathcal{F}$ denote the plurifine topology, the coarsest topology making every psh function continuous. Topological notions relative to the plurifine topology will be prefixed by \emph{$\mathcal{F}$-}. In this section, we assume some familiarity with fine pluripotential theory. The results that we need are proved in several articles; see, for example, \cite{EMW06,EKFW11,EMthesis}. A good reference is the book \cite{EKF}.

Recall that a function $v\colon V\to[-\infty,\infty)$ on an $\mathcal{F}$-open set $V\subseteq \mathbb{C}^n$ is \emph{(weakly) $\mathcal{F}$-psh} if it is not identically $-\infty$ on each $\mathcal{F}$-component, it is $\mathcal{F}$-upper semicontinuous and, for every complex affine line $\Lambda$, its restriction to each fine component of $V\cap\Lambda$ is finely subharmonic or identically $-\infty$. The $\mathcal{F}$-psh functions satisfy the usual sheaf property, and if $V$ is a connected open set, they are the same as the usual notion of psh functions. See \cite[Theorem~5.4.2(i), Proposition~5.4.6]{EKF}.

\begin{lemma}\label{lma:Fopen_notnull}
Every nonempty $\mathcal{F}$-open set has positive Lebesgue measure.
\end{lemma}
\begin{proof}
By the Bedford--Taylor base theorem \cite[Theorem~5.2.8]{EKF}, it suffices to show that a nonempty set of the form
\[
\{x\in B:h(x)>0\}
\]
has positive measure, where $B\subseteq \mathbb{C}^n$ is the unit ball and $h\in \PSH(B)$. If this set were null, then $h\leq 0$ almost everywhere and hence everywhere by \cite[Proposition~1.2.6]{XiaLectures}. This is a contradiction.
\end{proof}

For a polydisc $B$, let $2B$ denote the concentric polydisc with twice its polyradii. The following pluripolarity criterion is the main new ingredient in this paper:
\begin{theorem}\label{thm:plurifine-zero-set}
Let $m,n\geq1$. Let $B\subseteq\mathbb C^m$ be a polydisc and let $W\subseteq\mathbb C^m$ be open with $\overline{2B}\Subset W$. Let $G\subseteq\mathbb C^n$ be open and let $V\subseteq G$ be $\mathcal{F}$-open. For every $\eta\in G$, let $\mathcal H_\eta$ be a complex Hilbert space and $F_\eta\colon W\to\mathcal H_\eta$ a holomorphic map. Assume that
\[
 U(w,\eta)\coloneqq\log\|F_\eta(w)\|_{\mathcal H_\eta}^2 \in\PSH(W\times G),
\]
and, for some $M>0$, that
\begin{equation}\label{eq:Ulower}
 U>-M\quad\textup{on }\partial B\times V, \qquad \sup_{\overline{2B}\times V}U<\infty.
\end{equation}
If a conull set $G_0\subseteq G$ (i.e. the complement $G\setminus G_0$ is Lebesgue null) satisfies
\[
 F_\eta(w)\ne 0 \qquad(\eta\in G_0,\ w\in\overline{2B}),
\]
then
\[
 Z\coloneqq \left\{\eta\in V:F_\eta^{-1}(0)\cap B\ne\varnothing \right\}
\]
is pluripolar.
\end{theorem}
For the elementary facts about Hilbert-space-valued holomorphic functions, we refer to \cite[Chapter~3]{Rud91}.

\begin{proof}
We may assume that $V\ne\varnothing$. By \cite[Corollary~5.2.21]{EKF},
$V$ has at most countably many $\mathcal{F}$-connected components, all of
which are $\mathcal{F}$-open. We may therefore work on one component at a
time and assume that $V$ is $\mathcal{F}$-connected.

\noindent\textbf{Step 1}.
Put $C\coloneqq\sup_{\overline{2B}\times V}U$.
Embed $\overline{2B}$ in the standard affine chart of $\mathbb{P}^m$, and let $\omega_{\mathrm{FS}}$ be the Fubini--Study form.
Take $h\in C^\infty(\mathbb{P}^m)$ such that
\[
  h<-M-3\quad\text{near }\partial B, \qquad h>C+3\quad\text{near }\partial(2B).
\]
Choose an integer $d_0$ large enough that $h$ is $\omega$-psh for $\omega\coloneqq d_0\omega_{\mathrm{FS}}$ and $K_{\mathbb{P}^m}\otimes\mathcal{O}_{\mathbb{P}^m}(md_0)$ is globally generated. Let $L\coloneqq\mathcal{O}_{\mathbb{P}^m}(d_0)$ with its Fubini--Study metric $h_L$ of curvature $\omega$, put
\[
 \mathcal S\coloneqq
 \mathrm{H}^0(\mathbb{P}^m,K_{\mathbb{P}^m}\otimes L^{\otimes m}),
\]
and fix a basis $s_1,\ldots,s_\mu$ of $\mathcal S$.

For $\eta\in V$, set
\[
 \Psi_\eta\coloneqq
 \begin{cases}
 U_\eta,&\text{on }B,\\
 \rmax_{\varepsilon}(U_\eta,h),&\text{on }2B\setminus B,\\
 h,&\text{on }\mathbb{P}^m\setminus2B.
\end{cases}
\]
Here $\rmax_{\varepsilon}$ denotes the usual regularized maximum.
The boundary lower bound implies that $F_\eta$ is nowhere zero on
$\partial B$, so $U_\eta$ is continuous on a neighborhood of $\partial B$.
Near $\partial B$ we have $U_\eta>h+3$, whereas near $\partial(2B)$ we have
$h>U_\eta+3$. Thus, for all sufficiently small $\varepsilon>0$,
independently of $\eta$, the regularized maximum agrees with $U_\eta$ near
$\partial B$ and with $h$ near $\partial(2B)$. The three pieces therefore
glue and $\Psi_{\eta}\in \PSH(\mathbb{P}^m,\omega)$. Fix such an
$\varepsilon$ for the rest of the proof.

\noindent\textbf{Step 2}.
Fix $\eta_0\in V$. By~\cite[Corollary~5.2.10]{EKF}, there is a compact set $K$ with
\[
 \eta_0\in \operatorname{int}_{\mathcal{F}}K,\qquad K\subseteq V.
\]
Here $\operatorname{int}_{\mathcal{F}}K$ denotes the plurifine interior of $K$.

Choose standard radial convolutions $U_k\downarrow U$, smooth and psh on a fixed neighborhood of $\overline{2B}\times K$. Then $U_k>-M>h+2$ on $\partial B\times K$. Here we have omitted the obvious pullback notation.
The compact sets
\[
 \{U_k\ge C+1\}\cap(\partial(2B)\times K)
\]
are nested with empty intersection because $U_k\downarrow U\le C$; hence one
of them is empty. After discarding finitely many terms, $U_k<C+1$ on $\partial(2B)\times K$. Therefore,
\[
 U_k>h+2 \quad \textup{on }\partial B\times K,\qquad   U_k<h-2\quad\textup{on }\partial(2B)\times K.
\]
These estimates still hold if we replace $K$ by a sufficiently small open neighborhood $O_k \Subset G$.
Let $p_1\colon\mathbb{P}^m\times O_k\to\mathbb{P}^m$ be the projection. The function
\[
 \Psi_k\coloneqq
 \begin{cases}
 U_k,&\textup{on } B\times O_k,\\
 \rmax_{\varepsilon}(U_k,h),&\textup{on } (2B \setminus B)\times O_k,\\
 h,&\textup{on }(\mathbb{P}^m\setminus 2B)\times O_k
\end{cases}
\]
is smooth and $p_1^*\omega$-psh on $\mathbb{P}^m\times O_k$. Moreover,
\[
 \Psi_k(p,\eta)\downarrow\Psi_\eta(p)\qquad((p,\eta)\in\mathbb{P}^m\times K).
\]
For any $\omega$-psh function $\Theta$ on $\mathbb{P}^m$, set
\[
 q_\Theta(s)\coloneqq \mathrm i^{m^2}\int_{\mathbb{P}^m}|s|_{h_L^m}^2 \mathrm{e}^{-m\Theta}\in[0,\infty],
 \qquad s\in\mathcal S.
\]
Here $\mathrm i^{m^2}|s|_{h_L^m}^2$ is regarded as a positive Radon measure on $\mathbb{P}^m$.

When $q_\Theta$ is finite on $\mathcal S$, let $\det q_\Theta$ denote the determinant of its Gram matrix in the fixed basis. For $\eta\in O_k$, put
\[
 q_{k,\eta}\coloneqq q_{\Psi_k(\,\cdot\,,\eta)},\qquad
 \chi_k(\eta)\coloneqq-\log\det q_{k,\eta}.
\]
The form $q_{k,\eta}$ is finite and positive definite. By~\cite[Theorem~1.2]{Bern09}, the function $\chi_k$ is psh on $O_k$.

For $\eta\in V$, put $q_\eta\coloneqq q_{\Psi_\eta}$ and
\[
 \chi(\eta)\coloneqq
 \begin{cases}
  -\log\det q_\eta,&q_\eta\text{ is finite on }\mathcal S,\\
  -\infty,&\text{otherwise}.
 \end{cases}
\]
For $\eta\in\operatorname{int}_{\mathcal{F}}K$, monotone convergence gives $q_{k,\eta}(s)\uparrow q_\eta(s)$ for every $s\in\mathcal S$. If $q_\eta$ is finite, the Gram matrices of $q_{k,\eta}$ converge to that of $q_\eta$. Otherwise one eigenvalue of the Gram matrix of $q_{k,\eta}$ tends to infinity, while $q_{k,\eta}\ge q_{1,\eta}>0$, so $\det q_{k,\eta}\to\infty$. Thus in both cases
\[
 \chi_k\downarrow\chi\qquad\text{on }\operatorname{int}_{\mathcal{F}} K.
\]
By \cite[Theorem~5.4.2(i) and (iv)]{EKF}, $\chi$ is $\mathcal{F}$-psh on
$\operatorname{int}_{\mathcal{F}}K$ or identically $-\infty$ there. The
same conclusion holds on an $\mathcal{F}$-neighborhood of every point of
$V$. By the sheaf property and the $\mathcal{F}$-connectedness of $V$,
either $\chi$ is $\mathcal{F}$-psh on $V$, or
$\chi\equiv-\infty$ on $V$.

By~\cref{lma:Fopen_notnull}, $V$ has positive Lebesgue measure. Since
$G\setminus G_0$ is null, $V\cap G_0$ is nonempty; moreover,
$V\cap G_0\subseteq V\setminus Z$. If $\eta\in V\setminus Z$, the
inequality $U_\eta>-M$ on $\partial B$ and the definition of $Z$ show that
$F_\eta$ is nowhere zero on $\overline{B}$. Thus $U_\eta$ is bounded on
$\overline{B}$, $\Psi_\eta$ is bounded on $\mathbb{P}^m$, and $q_\eta$ is
finite and positive definite. Therefore $\chi$ is finite on
$V\setminus Z$. Since $V\setminus Z\ne\varnothing$, the second alternative
above is impossible, and hence $\chi$ is $\mathcal{F}$-psh.

\noindent\textbf{Step 3}.
By~\cite[Theorem~5.4.3(d)]{EKF}, the set
$\{\eta\in V:\chi(\eta)=-\infty\}$ is pluripolar.

Take $\eta\in Z$ and $w_*\in B$ with $F_\eta(w_*)=0$. A local bound for $\nabla F_\eta$, followed by integration along line segments, gives
\[
 \|F_\eta(w)\|_{\mathcal H_\eta}\le C_*|w-w_*|,
 \qquad U_\eta(w)\le2\log|w-w_*|+O(1).
\]
Since $K_{\mathbb{P}^m}\otimes L^{\otimes m}$ is globally generated, choose
$s\in\mathcal S$ with $s(w_*)\ne0$. As $\Psi_\eta=U_\eta$ on $B$, there are
$r>0$ and $c_*>0$ such that the displayed estimate and a local trivialization
give
\[
q_\eta(s)=\mathrm i^{m^2}\int_{\mathbb{P}^m}|s|_{h_L^m}^2 \mathrm{e}^{-m\Psi_\eta}\ge c_*\int_{|w-w_*|<r}|w-w_*|^{-2m}\,\mathrm dV=\infty.
\]
Thus $Z$ is contained in $\{\eta\in V:\chi(\eta)=-\infty\}$ and hence is
pluripolar.
\end{proof}

We next record two auxiliary lemmata for the projected zero-set argument.

\begin{lemma}\label{lem:Fopen-boundary-minimum}
Let $B\subseteq\mathbb C^m$ be a polydisc and let $W\subseteq\mathbb C^m$ be
open with $\overline{2B}\Subset W$. Let $G\subseteq\mathbb C^n$ be a domain.
For every $\eta\in G$, let $\mathcal H_\eta$ be a complex Hilbert space and
let $F_\eta\colon W\to\mathcal H_\eta$ be holomorphic. Assume that
\[
 U(w,\eta)\coloneqq\log\|F_\eta(w)\|_{\mathcal H_\eta}^2
 \in\PSH(W\times G).
\]
Then, for every $r>0$, the set
\[
 V_r\coloneqq
 \left\{\eta\in G:
 \min_{w\in\partial B}\|F_\eta(w)\|_{\mathcal H_\eta}>r
 \right\}
\]
is $\mathcal{F}$-open.
\end{lemma}

\begin{proof}
Fix $r>0$ and $\eta_0\in V_r$ and put
\[
 \beta_0\coloneqq
 \min_{w\in\partial B}\|F_{\eta_0}(w)\|_{\mathcal H_{\eta_0}},
 \qquad \tau\coloneqq\frac{\beta_0-r}{3}.
\]
Choose open neighborhoods
$\partial B\subseteq A_1\Subset A_2\Subset W$. After shrinking an ordinary
neighborhood $O\Subset G$ of $\eta_0$, local upper bounds for $U$ give
a uniform bound for $\|F_\eta\|$ on $A_2$, for $\eta\in O$. The
Hilbert-valued Cauchy estimates then give a uniform bound for
$\nabla F_\eta$ on $A_1$; hence the maps $F_\eta$, $\eta\in O$, are
equicontinuous on $\partial B$. Choose a finite net
$\{w_1,\ldots,w_s\}\subseteq\partial B$ so fine that, for every $\eta\in O$
and $w\in\partial B$, some $j$ satisfies
\[
 \|F_\eta(w)-F_\eta(w_j)\|_{\mathcal H_\eta}<\tau.
\]
For each $j$, the function $U(w_j,\cdot)$ belongs to $\PSH(G)$ because it
is finite at $\eta_0$. Hence
\[
 O\cap\bigcap_{j=1}^s
 \{\eta\in G:U(w_j,\eta)>2\log(\beta_0-\tau)\}
\]
is an $\mathcal{F}$-neighborhood of $\eta_0$. On this neighborhood the
triangle inequality gives
\[
 \min_{w\in\partial B}\|F_\eta(w)\|_{\mathcal H_\eta}
 >\beta_0-2\tau>r.
\]
Thus $V_r$ is $\mathcal{F}$-open.
\end{proof}

\begin{lemma}[Pluripolar Fubini lemma]\label{lem:pluripolar-Fubini}
Let $T\subseteq\mathbb C^d$ and $D\subseteq\mathbb C^q$, $q\geq1$, be
domains, and let $E\subseteq T\times D$ be pluripolar. Then
\[
 \left\{\theta\in T:E_\theta\text{ is non-pluripolar in }D\right\},
 \qquad
 E_\theta\coloneqq\{c\in D:(\theta,c)\in E\},
\]
is pluripolar in $T$.
\end{lemma}

\begin{proof}
By Josefson's theorem \cite{Jos78}, there is a function $\psi\in\PSH(T\times D)$ such that $E\subseteq\{\psi=-\infty\}$. If $E_\theta$ is non-pluripolar, then $\psi(\theta,\cdot)\equiv-\infty$ on $D$.

Choose $(\theta_0,c_0)\in T\times D$ with $\psi(\theta_0,c_0)>-\infty$. The function $\theta\mapsto\psi(\theta,c_0)$ is psh on $T$. Every $\theta$ for which $E_\theta$ is non-pluripolar belongs to its polar locus, which proves the lemma.
\end{proof}

\begin{corollary}\label{cor:projected-zero-set}
Let $W\Subset W'\subseteq\mathbb C^m$ and $G\Subset G'\subseteq\mathbb C^n$, where $m,n\geq1$, $W$ and $G$ are domains, and $W',G'$ are open. For every $\eta\in G'$, let $\mathcal H_\eta$ be a complex Hilbert space and let $F_\eta\colon W'\to\mathcal H_\eta$ be holomorphic. Assume that
\[
 U(w,\eta)\coloneqq\log\|F_\eta(w)\|_{\mathcal H_\eta}^2
 \in\PSH(W'\times G')
\]
and that
\[
 Z_\eta\coloneqq F_\eta^{-1}(0)\cap W
\]
is analytic in $W$ for every $\eta\in G$. If $Z_\eta=\varnothing$ for
$\eta$ in a conull subset $G_0\subseteq G$, then
\[
 \mathcal P_U\coloneqq\{\eta\in G:Z_\eta\ne\varnothing\}
\]
is pluripolar.
\end{corollary}

\begin{proof}
Call a polydisc rational if its center lies in $\mathbb Q(\mathrm i)^m$ and all its polyradii are positive rational numbers. Let $\mathscr B$ be thecountable family of rational polydiscs $B$ satisfying $\overline{2B}\Subset W$.

We first detect isolated zeros. For $B\in\mathscr B$ and
$M\in\mathbb Z_{\geq1}$, put
\[
 V^0_{B,M}\coloneqq \left\{\eta\in G:\min_{w\in\partial B}\|F_\eta(w)\|_{\mathcal H_\eta}>\mathrm e^{-M/2}\right\}
\]
and
\[
 E^0_{B,M}\coloneqq \left\{\eta\in V^0_{B,M}:Z_\eta\cap B\ne\varnothing\right\}.
\]
By \cref{lem:Fopen-boundary-minimum}, $V^0_{B,M}$ is $\mathcal{F}$-open.
Note that $U$ is bounded above on $\overline{2B}\times G$ because $\overline{2B}\times\overline G\Subset W'\times G'$, and
$U>-M$ on $\partial B\times V^0_{B,M}$. Moreover, $F_\eta$ has no zero on $\overline{2B}$ for $\eta\in G_0$. Hence \cref{thm:plurifine-zero-set} shows that every $E^0_{B,M}$ is pluripolar.

We next treat positive-dimensional zeros. For a nonempty set $I\subseteq\{1,\ldots,m\}$, let
\[
 \pi_I\colon\mathbb C^m\longrightarrow\mathbb C^{|I|}
\]
be the corresponding coordinate projection. Given $B\in\mathscr B$, set $D_{I,B}\coloneqq\pi_I(B)$. On the enlarged base $G\times D_{I,B}$ consider
\[
 H^I_{\eta,c}(w)\coloneqq \left(F_\eta(w),\pi_I(w)-c\right) \in\mathcal H_\eta\oplus\mathbb C^{|I|}
\]
and
\[
 U_I(w,\eta,c)
 \coloneqq\log\|H^I_{\eta,c}(w)\|^2  =\log\left(\mathrm e^{U(w,\eta)}+|\pi_I(w)-c|^2\right).
\]
Note that $U_I$ is psh.
For $M\in\mathbb Z_{\geq1}$, define
\[
 V_{I,B,M}\coloneqq
 \left\{(\eta,c)\in G\times D_{I,B}:
 \min_{w\in\partial B}\|H^I_{\eta,c}(w)\|>\mathrm e^{-M/2}
 \right\}
\]
and
\[
 E_{I,B,M}\coloneqq
 \left\{(\eta,c)\in V_{I,B,M}:
 Z_\eta\cap B\cap\pi_I^{-1}(c)\ne\varnothing
 \right\}.
\]
The set $V_{I,B,M}$ is $\mathcal{F}$-open by \cref{lem:Fopen-boundary-minimum}. The preceding upper bound for $U$, together with the boundedness of $D_{I,B}$, shows that $U_I$ is bounded above on $\overline{2B}\times G\times D_{I,B}$; its definition gives $U_I>-M$ on $\partial B\times V_{I,B,M}$. The set $G_0\times D_{I,B}$ is conull in the enlarged base, and $H^I_{\eta,c}$ has no zero on $\overline{2B}$ there. Thus \cref{thm:plurifine-zero-set} gives that $E_{I,B,M}$ is pluripolar in $G\times D_{I,B}$.
By \cref{lem:pluripolar-Fubini}, the set
\[
 A_{I,B,M}\coloneqq \left\{\eta\in G: (E_{I,B,M})_\eta\text{ is non-pluripolar in }D_{I,B} \right\}
\]
is pluripolar in $G$.

We claim that
\begin{equation}\label{eq:coordinate-collar-cover}
 \mathcal P_U= \bigcup_{B,M}E^0_{B,M}\cup \bigcup_{\varnothing\ne I\subseteq\{1,\ldots,m\}}\ \bigcup_{B,M}A_{I,B,M},
\end{equation}
where $B\in\mathscr B$ and $M\in\mathbb Z_{\geq1}$. The right-hand side is contained in $\mathcal P_U$ by construction. Conversely, fix $\eta\in\mathcal P_U$. Give $Z_\eta$ its reduced structure, choose $p\in(Z_\eta)_{\mathrm{reg}}$, and put $q=\dim_pZ_\eta$.

If $q=0$, then $p$ is isolated in $Z_\eta$. Choose $B\in\mathscr B$ with $p\in B$ so small that $Z_\eta\cap\overline{2B}=\{p\}$. The minimum of $\|F_\eta\|$ on $\partial B$ is positive, so $\eta\in E^0_{B,M}$ for some $M$.

Suppose that $q\geq1$. Since $Z_{\eta}$ is smooth at $p$, there is a set $I$, $|I|=q$, such that $\pi_I|_{Z_\eta}$ is biholomorphic near $p$. On a sufficiently small product
polydisc $P\Subset W$ about $p$, the set $Z_\eta\cap P$ is the graph of a
holomorphic map over the $I$-coordinates. Choose $B\in\mathscr B$ with
$p\in B$ and $\overline{2B}\Subset P$. There is a nonempty polydisc
$C\Subset D_{I,B}$ such that, for every $c\in\overline C$, the corresponding
point of the graph lies in $B$. Consequently,
\[
 H^I_{\eta,c}(w)\ne0
 \qquad((w,c)\in\partial B\times\overline C).
\]
Compactness gives a number $\delta>0$ such that
\[
 \min_{(w,c)\in\partial B\times\overline C}
 \|H^I_{\eta,c}(w)\|\geq\delta.
\]
Choose $M$ with $\mathrm e^{-M/2}<\delta$. Then
$C\subseteq(E_{I,B,M})_\eta$. Since a nonempty open set is non-pluripolar,
$\eta\in A_{I,B,M}$. This proves \eqref{eq:coordinate-collar-cover}.

There are only finitely many coordinate sets $I$ and countably many pairs
$(B,M)$. The right-hand side of \eqref{eq:coordinate-collar-cover} is
therefore pluripolar, and so is $\mathcal P_U$.
\end{proof}

\section{The local analytic Bertini theorem}\label{sec:analytic-bertini}

Fix $m,n$, $X$ and $\Phi$ as in~\eqref{eq:X_Phi}. The notations $J$, $J_{\eta}$ and $I_{\eta}$ defined in \eqref{eq:JJetaIeta} will be used.

The cases $m=0$ or $n=0$ are both trivial, we may therefore assume throughout the remainder of the proof that $m,n>0$.

\subsection{Jet bundles and notation}

All constructions in this section are local. Fix polydiscs
\[
 D\Subset\Delta_z^m,\qquad G\Subset\Delta_\eta^n.
\]
Let $N\in\mathbb N$. For $x\in D$, put
\[
 R_x\coloneqq\mathcal{O}_{D,x},\qquad \mathfrak m_x\coloneqq\{f\in R_x:f(x)=0\}.
\]
Let $\Jet^N\to D$ be the holomorphic bundle of jets of order at most $N$, characterized by
\[
 \Jet_x^N\coloneqq R_x/\mathfrak m_x^{N+1},\qquad \nu_N\coloneqq\rk\Jet^N=\binom{m+N}{m}.
\]
If $f$ is holomorphic near $x$, then
\[
 j_x^N f\coloneqq f_x\bmod\mathfrak m_x^{N+1}\in \Jet_x^N
\]
denotes its order-$N$ jet. For each multi-index $\alpha\in \mathbb{N}^m$, set
\[
 e_{\alpha,x}\coloneqq j_x^N\Bigl((z-x)^\alpha\Bigr)\in \Jet_x^N,
 \qquad
 D_x^\alpha(j_x^N f)\coloneqq \frac{1}{\alpha!}\partial_z^\alpha f(x).
\]
The vectors $(e_{\alpha,x})_{|\alpha|\le N}$ and covectors $(D_x^\alpha)_{|\alpha|\le N}$ are dual holomorphic frames of $\Jet^N$ and $(\Jet^N)^*$, respectively, and
\[
 j_x^N f=\sum_{|\alpha|\le N}D_x^\alpha f\,e_{\alpha,x},
\]
where we abbreviate $D_x^\alpha(j_x^N f)$ to $D_x^\alpha f$.

Over $D\times G$, the symbols $\Jet^N$ and $(\Jet^N)^*$ also denote their pullbacks by the vertical projection $\operatorname{pr}_z\colon D\times G\to D$.

We identify $X_\eta$ with $\Delta_z^m$ by the vertical coordinate and hence $R_x$ with $\mathcal{O}_{X_\eta,(x,\eta)}$. In particular, for the stalk ideals $I_{\eta,x}\coloneqq (I_\eta)_{(x,\eta)}$ and $J_{\eta,x}\coloneqq (J_\eta)_{(x,\eta)}$ in $R_x$, put
\[
 E_N^I(x,\eta)\coloneqq
 \frac{I_{\eta,x}+\mathfrak m_x^{N+1}}{\mathfrak m_x^{N+1}},\qquad
 E_N^J(x,\eta)\coloneqq
 \frac{J_{\eta,x}+\mathfrak m_x^{N+1}}{\mathfrak m_x^{N+1}}.
\]
By~\eqref{eq:OT-inclusion},
\[
 E_N^I(x,\eta)\subseteq E_N^J(x,\eta).
\]

For $\eta\in G$, set
\[
 H_\eta\coloneqq A^2(D,\mathrm{e}^{-\Phi_\eta})=\left\{ f\in \mathcal{O}(D): \|f\|_{H_{\eta}}^2<\infty\right\},
 \qquad
 \|f\|_{H_\eta}^2=\int_D|f|^2\mathrm{e}^{-\Phi_\eta}\,\mathrm dV.
\]
A fiber restriction $\Phi_\eta$ may be identically $-\infty$; in that case
$H_\eta=\{0\}$. Otherwise, for any $x\in D$, choose a polydisc
$B_x\Subset D$ containing $x$. Then
\begin{equation}\label{eq:weighted-unweighted}
 \|f\|_{L^2(B_x)}^2
 \le \mathrm{e}^{\sup_{B_x}\Phi_\eta}\|f\|_{H_\eta}^2,
\end{equation}
so Cauchy inequalities make every order-$N$ jet functional continuous on
$H_{\eta}$. Thus, in either case, for any $x\in D$, the coordinate
functionals
\[
 \ell_{\alpha,\eta}(x)\in H_\eta^*,\qquad \ell_{\alpha,\eta}(x)(f)\coloneqq D_x^\alpha f, \qquad |\alpha|\le N,
\]
are well defined.

For a Hilbert space $H$ and $d\ge1$, let $H^{\widehat\otimes d}$ be its completed $d$-fold tensor product and let $\widehat{\bigwedge}^{\,d}H$ be the closed alternating subspace. Let $\mathfrak S_d$ denote the permutation group of $\{1,\ldots,d\}$. For $\ell_1,\ldots,\ell_d\in H$, our normalized wedge convention is
\[
 \ell_1\wedge\cdots\wedge\ell_d
 \coloneqq \frac{1}{\sqrt{d!}}
   \sum_{\sigma\in\mathfrak S_d}\operatorname{sgn}(\sigma)
   \ell_{\sigma(1)}\otimes\cdots\otimes\ell_{\sigma(d)}.
\]
Thus
\[
 \|\ell_1\wedge\cdots\wedge\ell_d\|^2=\det(\langle\ell_i,\ell_j\rangle)_{i,j=1}^d.
\]

\subsection{Jet images of the fiber and restricted ambient ideals}

\begin{lemma}\label{lem:jet-image}
For every $\eta\in G$ and $x\in D$,
\begin{equation}\label{eq:I-jet-image}
 \im \left(j_x^N\colon H_\eta\to \Jet_x^N\right)=E_N^I(x,\eta).
\end{equation}
\end{lemma}

\begin{proof}
Only surjectivity in~\eqref{eq:I-jet-image} needs proof. Let
$\mathcal M_{x,N+1}\subseteq\mathcal{O}_{\Delta_z^m}$ be the order-$N$
fat-point ideal. The exact sequence
\[
 0\longrightarrow I_\eta\cap\mathcal M_{x,N+1}
 \longrightarrow I_\eta
 \longrightarrow
 \frac{I_\eta+\mathcal M_{x,N+1}}{\mathcal M_{x,N+1}}
 \longrightarrow0
\]
has coherent kernel and a coherent rightmost term supported at $x$. Since
$\Delta_z^m$ is Stein, Cartan's Theorem B gives
$\mathrm{H}^1(\Delta_z^m,I_\eta\cap\mathcal M_{x,N+1})=0$. Hence the induced
map on global sections is surjective, and the global sections of its
rightmost term identify canonically with $E_N^I(x,\eta)$. Thus every
prescribed jet lifts to a section of $I_\eta$ on $\Delta_z^m$. Such a
section is locally square integrable with weight
$\mathrm{e}^{-\Phi_\eta}$; since $\overline D\Subset\Delta_z^m$, its
restriction belongs to $H_\eta$. This proves~\eqref{eq:I-jet-image}.
\end{proof}

Since $X$ is Stein, Cartan's Theorem A and compactness of
$\overline D\times\overline G$ give an integer $r\ge1$ and sections
$g_1,\ldots,g_r\in\Gamma(X,J)$ that generate $J$ on a neighborhood of
$\overline D\times\overline G$. For $1\le i\le r$ and $|\alpha|\le N$,
define the moving-center columns
\[
 s_{i,\alpha}(x,\eta) \coloneqq j_x^N\big((z-x)^\alpha g_i(z,\eta)\big)\in \Jet_x^N.
\]
They give a bundle morphism
\[
 \mathcal A_N\colon \mathcal{O}_{D\times G}^{\,r\nu_N}\longrightarrow\Jet^N, \qquad (c_{i,\alpha})\longmapsto\sum_{i,\alpha}c_{i,\alpha}s_{i,\alpha},
\]
whose pointwise image is
\begin{equation}\label{eq:J-jet-image}
 \im\mathcal A_N(x,\eta)=E_N^J(x,\eta)
 =\frac{J_{\eta,x}+\mathfrak m_x^{N+1}}
        {\mathfrak m_x^{N+1}}.
\end{equation}
Indeed, if $f=\sum_i b_i g_i$, then modulo $\mathfrak m_x^{N+1}$ each
holomorphic coefficient $b_i$ may be replaced by its degree-$N$ Taylor
polynomial. In normalized jet
coordinates, the matrix entries are
\[
 D_x^\beta\Bigl(s_{i,\alpha}(x,\eta)\Bigr)
 =
 \begin{cases}
  (D_z^{\beta-\alpha}g_i)(x,\eta),&\beta\ge_{\text{comp}}\alpha,\\
  0,&\text{otherwise}.
 \end{cases},
\]
where $\beta\ge_{\text{comp}}\alpha$ means componentwise inequality.
Thus $\mathcal A_N$ is holomorphic.

\subsection{The rank detector and its Bergman realization}

For $1\le\rho\le\nu_N$, put
\[
 \Omega_{N,\rho}\coloneqq  \left\{(x,\eta)\in D\times G:\dim_{\mathbb C}E_N^J(x,\eta)\ge\rho\right\}.
\]
In view of \eqref{eq:J-jet-image}, its complement is the common zero set of the $\rho\times\rho$ minors of $\mathcal A_N$, so $\Omega_{N,\rho}$ is open.

Order multi-indices lexicographically. For $1\le\rho\le\nu_N$, let
\[
 \mathfrak T_{N,\rho} \coloneqq \Bigl\{(\alpha_1,\ldots,\alpha_\rho)\in(\mathbb N^m)^\rho: |\alpha_i|\le N\ (1\le i\le\rho),\ \alpha_1<\cdots<\alpha_\rho\Bigr\}.
\]
For $\boldsymbol\alpha=(\alpha_1,\ldots,\alpha_\rho)\in\mathfrak T_{N,\rho}$ and $x\in D$, put
\[
 F_{\boldsymbol\alpha,\eta}(x) \coloneqq \ell_{\alpha_1,\eta}(x)\wedge\cdots\wedge\ell_{\alpha_\rho,\eta}(x).
\]
Using the finite Hilbert direct sum, set
\[
 \mathcal H_{N,\rho,\eta}
 \coloneqq \bigoplus_{\boldsymbol\alpha\in\mathfrak T_{N,\rho}} \widehat{\bigwedge}^{\,\rho}H_\eta^*,
 \qquad
 F_{N,\rho,\eta}(x)
 \coloneqq
 \Bigl(F_{\boldsymbol\alpha,\eta}(x)\Bigr)_{
   \boldsymbol\alpha\in\mathfrak T_{N,\rho}}
 \in \mathcal H_{N,\rho,\eta}.
\]
Set
\begin{equation}\label{eq:detector}
 U_{N,\rho}(x,\eta) \coloneqq\log\left\|F_{N,\rho,\eta}(x)\right\|_{\mathcal H_{N,\rho,\eta}}^2.
\end{equation}

\begin{proposition}\label{prop:detector}
For every $1\le\rho\le\nu_N$,
\[
 U_{N,\rho}\in\PSH(D\times G)\cup\{-\infty\},
\]
and
\begin{equation}\label{eq:detector-threshold}
 U_{N,\rho}(x,\eta)=-\infty \quad\Longleftrightarrow \quad \dim_{\mathbb C}E_N^I(x,\eta)<\rho.
\end{equation}
Consequently,
\begin{equation}\label{eq:open-rank-defect}
 \Omega_{N,\rho}\cap \left\{U_{N,\rho}=-\infty \right\} =\left\{(x,\eta): \dim_{\mathbb C}E_N^I(x,\eta)<\rho \le\dim_{\mathbb C}E_N^J(x,\eta)\right\}.
\end{equation}
For fixed $\eta\in G$, the polar locus of $U_{N,\rho}(\cdot,\eta)$ is analytic in $D$, and $F_{N,\rho,\eta}$ is holomorphic on $D$.
\end{proposition}

\begin{proof}
Fix $\rho\ge1$ and $\boldsymbol\alpha=(\alpha_1,\ldots,\alpha_\rho)
\in\mathfrak T_{N,\rho}$. On $D^\rho$, consider the operator
\[
\mathcal D_{\boldsymbol\alpha}
 \coloneqq \sum_{\sigma\in\mathfrak S_\rho}\operatorname{sgn}(\sigma)
 D_{z_1}^{\alpha_{\sigma(1)}}\cdots
 D_{z_\rho}^{\alpha_{\sigma(\rho)}}.
\]
Put
\[
 \psi(z_1,\ldots,z_\rho,\eta)
 \coloneqq\sum_{j=1}^\rho\Phi(z_j,\eta).
\]
\Cref{thm:fixed-functional} applies with $q=m\rho$, $s=n$, $V=D^\rho$, the fixed base polydisc $G$, and $\mathcal D=\mathcal D_{\boldsymbol\alpha}$. Multiplication of elementary tensors gives canonical unitary identifications
\[
 H_\eta^{\widehat\otimes\rho}
 \simeq A^2\!\left(D^\rho,\mathrm{e}^{-\sum_j\Phi_\eta(z_j)}\right),\qquad (H_\eta^*)^{\widehat\otimes\rho}
 \simeq (H_\eta^{\widehat\otimes\rho})^*,
\]
see~\cite[proof of Theorem~15.2(a)]{Dem12}. By density of elementary tensors, computing the functional $\mathcal D_{\boldsymbol\alpha}$ and using the normalized wedge convention gives the diagonal identity
\[
 K_{\mathcal D_{\boldsymbol\alpha},D^\rho}^{\psi_\eta}(x,\ldots,x)
 =\rho!\,\|F_{\boldsymbol\alpha,\eta}(x)\|^2.
\]
Indeed, under the preceding unitary identifications, evaluation of
$\mathcal D_{\boldsymbol\alpha}$ at $(x,\ldots,x)$ is the functional
$\sqrt{\rho!}\,F_{\boldsymbol\alpha,\eta}(x)$. Let
\[
 \delta\colon D\times G\longrightarrow D^\rho\times G,\qquad
 \delta(x,\eta)\coloneqq(x,\ldots,x,\eta).
\]
Thus, by~\cref{thm:fixed-functional},
\[
 u_{\boldsymbol\alpha}(x,\eta)
 \coloneqq\log \left\|F_{\boldsymbol\alpha,\eta}(x)\right\|^2
 =\left(\log K_{\mathcal D_{\boldsymbol\alpha},D^\rho}^{\psi_\eta}\right) \circ\delta(x,\eta)-\log\rho! \in\PSH(D\times G)\cup\{-\infty\}.
\]
Therefore so is
\[
 U_{N,\rho}=\log\left(\sum_{\boldsymbol\alpha\in\mathfrak T_{N,\rho}} \mathrm{e}^{u_{\boldsymbol\alpha}}\right).
\]

\Cref{lem:jet-image} gives
\[
 \dim_{\mathbb C}\operatorname{span}\left\{\ell_{\alpha,\eta}(x):|\alpha|\le N \right\}
 =\rk\left(j_x^N|_{H_\eta}\right)=\dim_{\mathbb C}E_N^I(x,\eta).
\]
Now $F_{N,\rho,\eta}(x)=0$ exactly when $\operatorname{span}\{\ell_{\alpha,\eta}(x):|\alpha|\le N\}$ has dimension less than $\rho$. This proves~\eqref{eq:detector-threshold}; the definition of $\Omega_{N,\rho}$ gives~\eqref{eq:open-rank-defect}.

Fix $\eta\in G$. For every $f\in H_\eta$, the scalar function $x\mapsto\ell_{\alpha,\eta}(x)(f)=D_x^\alpha f$ is holomorphic, so $x\mapsto\ell_{\alpha,\eta}(x)$ is holomorphic. It follows that $F_{N,\rho,\eta}$ is holomorphic.

It remains to show that $\{U_{N,\rho}(\cdot,\eta)=-\infty\}$ is analytic. Fix $x_0\in D$. By coherence, on a neighborhood $O_{x_0}\ni x_0$ choose generators $f_1,\ldots,f_c$ of $I_\eta$. The moving-center vectors
\[
 j_x^N\Bigl((z-x)^\gamma f_i(z)\Bigr), \qquad 1\le i\le c,\quad |\gamma|\le N,
\]
form the columns of a finite matrix holomorphic in $x\in O_{x_0}$. By the same Taylor-polynomial argument as in~\eqref{eq:J-jet-image}, the image of this matrix is $E_N^I(x,\eta)$. Hence~\eqref{eq:detector-threshold} identifies $\{U_{N,\rho}(\cdot,\eta)=-\infty\}\cap O_{x_0}$ with the locus where this matrix has rank $<\rho$. It is analytic. Such neighborhoods $O_{x_0}$ cover $D$.
\end{proof}

\subsection{Proof of the main theorem}
We now assemble the preceding rank detectors to prove \cref{thm:main}.

\begin{proof}[Proof of \cref{thm:main}]
Krull's intersection theorem gives
\[
\bigcap_{r\ge1}(I_{\eta,x}+\mathfrak m_x^r)=I_{\eta,x}
\]
and hence
\begin{equation}\label{eq:Krull}
 I_{\eta,x}\ne J_{\eta,x} \quad\Longleftrightarrow\quad E_N^I(x,\eta)\ne E_N^J(x,\eta) \text{ for some }N\in\mathbb N.
\end{equation}

For each $N\in\mathbb N$ and $1\le\rho\le\nu_N$, choose a countable family $\mathcal Q_{N,\rho}$ of product polydiscs
\[
 Q=W_Q\times G_Q\Subset\Omega_{N,\rho}
\]
that cover $\Omega_{N,\rho}$.

Let $\operatorname{pr}\colon D\times G\to G$ be the base projection and let
\begin{equation}\label{eq:BDG}
 \mathcal B(D,G)
 \coloneqq \left\{\eta\in G:I_{\eta,x}\ne J_{\eta,x} \text{ for some }x\in D \right\}.
\end{equation}
By~\eqref{eq:Krull} and the inclusion
$E_N^I(x,\eta)\subseteq E_N^J(x,\eta)$, a stalk difference occurs exactly
when, for some $N$ and $\rho$,
\[
 \dim_{\mathbb C}E_N^I(x,\eta)<\rho
 \leq\dim_{\mathbb C}E_N^J(x,\eta).
\]
Hence \eqref{eq:open-rank-defect} gives
\begin{equation}\label{eq:projected-incidence}
 \mathcal B(D,G)
 =\bigcup_{N\in\mathbb N}\ \bigcup_{\rho=1}^{\nu_N}\
   \bigcup_{Q\in\mathcal Q_{N,\rho}}
   \operatorname{pr}
   \Bigl(Q\cap\{U_{N,\rho}=-\infty\}\Bigr).
\end{equation}

Fix $N$, $\rho$, and $Q=W_Q\times G_Q\in\mathcal Q_{N,\rho}$. Choose a
product polydisc
\[
 Q'=W'_Q\times G'_Q
 \quad\text{with}\quad
 Q\Subset Q'\Subset\Omega_{N,\rho}.
\]
On $Q'$, define
\[
 \mathcal H_\eta\coloneqq \mathcal H_{N,\rho,\eta},\qquad
 F_\eta(w)\coloneqq F_{N,\rho,\eta}(w),\qquad
 U(w,\eta)\coloneqq U_{N,\rho}(w,\eta).
\]
Then
\begin{equation}\label{eq:pulled-detector}
  \begin{split}
 U\in\PSH\left(W'_Q\times G'_Q\right)\cup \{-\infty\},\qquad
 F_\eta\colon W'_Q\to\mathcal H_\eta\text{ is holomorphic},\\
 U(w,\eta)=\log\|F_\eta(w)\|_{\mathcal H_\eta}^2.
  \end{split}
\end{equation}
By~\cref{prop:detector}, for every $\eta\in G_Q$, the following set is
analytic in $W_Q$:
\[
 Z_\eta\coloneqq F_\eta^{-1}(0)\cap W_Q
 =\{w\in W_Q:U(w,\eta)=-\infty\}.
\]
Define the corresponding parameter locus by
\[
 \mathcal P_U\coloneqq \{\eta\in G_Q:Z_\eta\ne\varnothing\}.
\]
Put $G_{\mathrm{good}}\coloneqq G_Q\cap G_{\mathrm{ae}}$. Recall that $G_{\mathrm{ae}}$ is defined in \eqref{eq:ae-good}.
The set $G_{\mathrm{good}}$ is conull in $G_Q$ by~\eqref{eq:ae-good}, and the inclusion $W_Q\times G_Q\subseteq\Omega_{N,\rho}$ together with
$E_N^I(w,\eta)=E_N^J(w,\eta)$ for $w\in W_Q$ and
$\eta\in G_{\mathrm{good}}$ gives
\[
 Z_\eta=\varnothing\qquad(\eta\in G_{\mathrm{good}}).
\]
Consequently, $U$ is finite on $W_Q\times G_{\mathrm{good}}$, so the psh alternative in \eqref{eq:pulled-detector} holds. Applying \cref{cor:projected-zero-set} to $W_Q\Subset W'_Q$ and $G_Q\Subset G'_Q$ now shows that $\mathcal P_U$ is pluripolar. It follows that $\mathcal B(D,G)$ is pluripolar by \eqref{eq:projected-incidence}.

Choose polydiscs, indexed by $\ell\in\mathbb N$, such that
\[
 D_\ell\Subset\Delta_z^m,\qquad G_\ell\Subset\Delta_\eta^n,
\]
and the sets $D_\ell\times G_\ell$ cover $X$. The set
$P\coloneqq\bigcup_\ell \mathcal B(D_\ell,G_\ell)$ is pluripolar. If
$\eta\notin P$ and $x\in\Delta_z^m$, choose $\ell$ with
$(x,\eta)\in D_\ell\times G_\ell$. Then
$\eta\notin\mathcal B(D_\ell,G_\ell)$, so
$I_{\eta,x}=J_{\eta,x}$. Thus the two ideal sheaves agree on $X_\eta$.
\end{proof}

\section{Analytic Bertini theorem for holomorphic maps}
\label{sec:manifold-bertini}

Complex manifolds are assumed $\sigma$-compact. Recall that a subset $P$ of a complex manifold $Y$ is \emph{pluripolar}\footnote{Some authors prefer to say \emph{locally pluripolar}.} if for any $y\in Y$, there is an open neighborhood $U$ of $y$ and a psh function $\varphi$ on $U$ such that $\varphi|_{P\cap U}\equiv -\infty$.
Recall that a subset $A$ of a complex manifold $Y$ is \emph{negligible} if $A$ is contained in a countable union of locally closed complex submanifolds of $Y$ with empty interiors.

We first recall a well-known generic smoothness theorem:
\begin{theorem}\label{thm:analytic-generic-smoothness}
Let $f\colon Y\to Z$ be a holomorphic map between complex manifolds, and let
$C_f\subseteq Y$ be the non-smooth locus of $f$. Then $f(C_f)$ is negligible.
\end{theorem}
In particular, $f(C_f)$ is pluripolar.

\begin{proof}
We may assume that $Z$ is connected. Put $n=\dim Z$. Then
\[
 C_f=\{y\in Y:\rk_{\mathbb C}\mathrm d f_y<n\}.
\]
In particular, it is a closed analytic subset; give it the reduced
structure.

No local holomorphic section of $f$ has image contained in $C_f$. Indeed, if
$\sigma$ were such a section, then
\[
 \mathrm d f_{\sigma(z)}\circ\mathrm d\sigma_z=\mathrm{id},
\]
contrary to the definition of $C_f$. Hence \cite[Lemma~V.4.4]{BS76} shows that $f(C_f)$ is negligible.
\end{proof}

\begin{proof}[Proof of \Cref{cor:manifold-bertini}]
Let $C_f$ be the locus in \cref{thm:analytic-generic-smoothness}.
If $z\notin f(C_f)$, then $Y_z\subseteq Y\setminus C_f$, so
$\mathrm df$ is surjective along $Y_z$ and the fiber is a complex
submanifold.

Any point in $Y\setminus C_f$ admits an open neighborhood on which $f$ is isomorphic to the projection of a product chart.
Choose a countable family of such product charts covering $Y\setminus C_f$. Applying \cref{thm:main} on each chart gives a pluripolar exceptional subset of its base coordinate neighborhood. Taking the union with $f(C_f)$ gives the required pluripolar exceptional locus.
\end{proof}

\clearpage
\printbibliography
Mingchen Xia, \textsc{Institute of Geometry and Physics, USTC}\par\nopagebreak
  \textit{Email address}, \texttt{xiamingchen2008@gmail.com}\par\nopagebreak
  \textit{Homepage}, \url{https://mingchenxia.github.io/}.

\end{document}